\documentclass[10pt,reqno]{amsart}
  \usepackage{geometry}
\usepackage{amsmath, amsthm, amscd, amsfonts, amssymb, graphicx, color}
\usepackage[bookmarksnumbered, colorlinks, plainpages]{hyperref}

\newtheorem{theorem}{Theorem}[section]
\newtheorem{lemma}[theorem]{Lemma}
\newtheorem{proposition}[theorem]{Proposition}
\newtheorem{corollary}[theorem]{Corollary}
\theoremstyle{definition}
\newtheorem{definition}[theorem]{Definition}

\theoremstyle{remark}
\newtheorem{remark}[theorem]{Remark}
\numberwithin{equation}{section}

\input{mathrsfs.sty}
\begin{document}

\author[E. H. Benabdi and M. Barraa]{El Hassan Benabdi and Mohamed Barraa}

\title[Drazin invertibility of quaternionic linear operators]{Drazin invertibility of linear operators on quaternionic Banach spaces}

\address{Department of Mathematics, Faculty of Sciences Semlalia, University Cadi Ayyad, Marrakesh, Morocco
}
\email{elhassan.benabdi@gmail.com}
\email{barraa@uca.ac.ma}



\begin{abstract}
Let $A$ be a right linear operator on a two-sided quaternionic Banach space $X$. The paper studies the Drazin inverse for right linear operators on a quaternionic Banach space. It is shown that if $A$ is Drazin invertible then the Drazin inverse of $A$ is given by $f(A)$ where $f$ is $0$ in an axially symmetric neighborhood of $0$ and $f(q) = q^{-1}$ in an axially symmetric neighborhood of the nonzero spherical spectrum of $A$. Some results analogous to the ones concerning the Drazin inverse of operators on complex Banach spaces are proved in the quaternionic context.
\end{abstract}\maketitle
\section{Introduction and preliminaires}
We denote by $\mathbb{H}$ the algebra of quaternions, introduced by Hamilton in 1843. An element $q$ of $\mathbb{H}$ is of the form
 $$q = a + b\mathrm{i} + c\mathrm{j} + d\mathrm{k}; a,b,c,d\in\mathbb{R}$$
where $\mathrm{i},\mathrm{j}$ and $\mathrm{k}$ are imaginary units. By definition, they satisfy
$$\mathrm{i}^2=\mathrm{j}^2=\mathrm{k}^2=\mathrm{i}\mathrm{j}\mathrm{k}=-1.$$
Given $q=a+b\mathrm{i}+c\mathrm{j}+d\mathrm{k}$, then
\begin{itemize}
\item[•] the conjugate quaternion of $q$ is $\bar{q}=a-b\mathrm{i}-c\mathrm{j}-d\mathrm{k}$;
\item[•] the norm of $q$ is $\vert q\vert=\sqrt{q\bar{q}}=\sqrt{a^2+b^2+c^2+d^2}$;
\item[•] the real and the imaginary parts of $q$ are respectively $\text{Re}(q):=\frac{1}{2}(q+\bar{q})=a$ and $\text{Im}(q):=\frac{1}{2}(q-\bar{q})=b\mathrm{i} + c\mathrm{j} + d\mathrm{k}$.
\end{itemize}
The unit sphere of imaginary quaternions is given by
$$\mathbb{S}=\{q\in\mathbb{H}:q^2=-1\}.$$
Let $p$ and $q$ be two quaternions. $p$ and $q$  are said to be conjugated, if there is $s\in\mathbb{H}\setminus\{0\}$ such that $p=sqs^{-1}$. The set of all quaternions conjugated with $q$, is equal to the 2-sphere 
$$[q] =\{\text{Re}(q)+\vert\text{Im}(q)\vert j : j \in \mathbb{S}\}=\text{Re}(q)+\vert\text{Im}(q)\vert\mathbb{S}.$$
For every $j \in \mathbb{S}$, denote by $\mathbb{C}_j$ the real subalgebra of $\mathbb{H}$ generated by $j$; that is, 
$$\mathbb{C}_j:=\{u+vj\in\mathbb{H}: u,v\in\mathbb{R}\}.$$
We say that $U\subseteq\mathbb{H}$ is axially symmetric if $[q] \subset U$ for every $q \in U$.\\

For a thorough treatment of the algebra of quaternions $\mathbb{H}$, the reader is referred, for instance, to \cite{34}.
\begin{definition}[{\cite[Definition 2.1.2]{31}}, Slice hyperholomorphic functions] 
Let $U\subseteq\mathbb{H}$ be an axially symmetric open set and let $\mathcal{U}=\{(u,v)\in\mathbb{R}^2: u+v\mathbb{S}\subset U\}$.\\
A function $f : U\rightarrow\mathbb{H}$ is called a left slice function if there exist two functions $f_0, f_1 : \mathcal{U}\rightarrow\mathbb{H}$ such that:
$$f(q) = f_0(u, v) + jf_1(u, v)\;\;\; \text{ for every } q = u + vj \in U$$
and if $f_0, f_1$ satisfy the compatibility conditions
\begin{equation}
f_0(u,-v) = f_0(u,v),\;\;\;\;\;\; f_1(u,-v) = -f_1(u,v).\label{3eq6}
\end{equation}
If in addition $f_0$ and $f_1$ satisfy the Cauchy-Riemann equations
\begin{equation}
\begin{split}
\frac{\partial}{\partial u}f_0(u, v)-\frac{\partial}{\partial v}f_1(u, v)= 0,\\
\frac{\partial}{\partial v}f_0(u, v)+\frac{\partial}{\partial u}f_1(u, v)= 0,\label{3eq7}
\end{split}
\end{equation}
then $f$ is called left slice hyperholomorphic. We denote the set of all left slice functions on $U$ by $\mathcal{SF}_L(U)$ and the set of all left slice hyperholomorphic functions on $U$ by $\mathcal{SH}_L(U)$.\\
A function $f : U\rightarrow\mathbb{H}$ is called a right slice function if there exist two functions $f_0, f_1 : \mathcal{U}\rightarrow\mathbb{H}$ such that:
$$f(q) = f_0(u, v) + f_1(u, v)j\;\;\; \text{ for every } q = u + vj \in U$$
and if $f_0, f_1$ satisfy the compatibility conditions $(\ref{3eq6})$. If in addition $f_0$ and $f_1$ satisfy
the Cauchy-Riemann equations $(\ref{3eq7})$, then $f$ is called right slice hyperholomorphic. We denote the set of all right slice functions on $U$ by $\mathcal{SF}_R(U)$ and the set of all right slice hyperholomorphic functions on $U$ by $\mathcal{SH}_R(U)$.\\
If $f$ is a left $($or right$)$ slice function such that $f_0$ and $f_1$ are real-valued, then $f$ is called intrinsic. The set of all intrinsic slice functions on $U$ will be denoted by $\mathcal{FN}(U)$ and the set of all intrinsic slice hyperholomorphic functions on $U$ will be denoted by $\mathcal{N}(U)$.
\end{definition}
\begin{lemma}[{\cite[Lemma 2.1.6]{31}}, Splitting lemma] Let $U\subseteq\mathbb{H}$ be an axially symmetric open set and let $i,j\in\mathbb{S}$ with $ij=-ji$. If $f \in\mathcal{SH}_L(U)$, then the restriction $f_j = f\vert_{ U\cap\mathbb{C}_j}$ satisfies
$$\frac{1}{2}\left(\frac{\partial}{\partial u}f_j(z)+j\frac{\partial}{\partial v}f_j(z)\right)= 0,$$
for all $z = u + vj \in U\cap\mathbb{C}_j$. Hence
$$f_j(z) = F_0(z) + F_1(z)i$$
with holomorphic functions $F_0,F_1 : U\cap\mathbb{C}_j\rightarrow\mathbb{C}_j.$\\
If $f \in\mathcal{SH}_R(U)$, then the restriction $f_j = f\vert_{U\cap\mathbb{C}_j}$ satisfies
$$\frac{1}{2}\left(\frac{\partial}{\partial u}f_j(z)+\frac{\partial}{\partial v}f_j(z)j\right)= 0,$$
for all $z = u + vj \in U\cap\mathbb{C}_j$. Hence
$$f_j(z) = F_0(z) +i F_1(z)$$
with holomorphic functions $F_0,F_1 : U\cap\mathbb{C}_j\rightarrow\mathbb{C}_j.$
\end{lemma}
\begin{theorem}[{\cite[Theorem 2.1.21]{31}}, Cauchy's integral theorem]
Let $U\subseteq\mathbb{H}$ be open, $j\in\mathbb{S}$ and $D_j \subset U\cap \mathbb{C}_j$ be a bounded open subset of the complex plane $\mathbb{C}_j$ with $\overline{D_j}  \subset U\cap \mathbb{C}_j$ such that $\partial D_j$ is a finite union of piecewise continuously differentiable Jordan curves. Then for all $f \in\mathcal{SH}_L(U)$ and all $g \in\mathcal{SH}_R(U)$
$$\int_{\partial D_j}g(s)ds_jf(s)=0,$$
where $ds_j = ds(-j)$.
\end{theorem}
\begin{definition}[{\cite[Definition 2.1.23]{31}}]
We define the left slice hyperholomorphic Cauchy kernel as:
$$S_L^{-1}(s, q) :=-(q^2-2\text{Re}(s)q + \vert s\vert^2)^{-1}(q -\bar{s});\;\;\;\;\;\;\;\; q \notin [s],$$
and the right slice hyperholomorphic Cauchy kernel as:
$$S_R^{-1}(s, q) :=-(q -\bar{s})(q^2-2\text{Re}(s)q + \vert s\vert ^2)^{-1};\;\;\;\;\;\;\;\; q \notin [s].$$
\end{definition}
\begin{lemma}[{\cite[Lemma 2.1.27]{31}}]
Let $q, s\in\mathbb{H}$ with $s \notin [q]$.\\
The left slice hyperholomorphic Cauchy kernel $S_L^{-1}(s, q)$ is left slice hyperholomorphic in $q$ and right slice hyperholomorphic in $s$.\\
The right slice hyperholomorphic Cauchy kernel $S_R^{-1}(s, q)$ is left slice hyperholomorphic in $s$ and right slice hyperholomorphic in $q$.
\end{lemma}
\begin{definition}[{\cite[Definition 2.1.30]{31}}, Slice Cauchy domain]
An axially symmetric open set $U \subset\mathbb{H}$ is called a slice Cauchy domain if $U \cap\mathbb{C}_j$ is a Cauchy domain in $\mathbb{C}_j$ for every $j\in\mathbb{S}$. More precisely, $U$ is a slice Cauchy domain if for every $j\in\mathbb{S}$ the boundary $\partial(U\cap\mathbb{C}_j)$ of $U\cap\mathbb{C}_j$ is the union a finite number of nonintersecting piecewise continuously differentiable Jordan curves in $\mathbb{C}_j$.
\end{definition}
\begin{theorem}[{\cite[Theorem 2.1.32]{31}}, Cauchy's formulas]
Let $U \subset\mathbb{H}$ be a bounded slice Cauchy domain, let $j\in\mathbb{S}$, and set $ds_j = ds(-j)$. If $f$ is a left slice hyperholomorphic function on a set that contains $\overline{U}$, then
$$f(q)=\frac{1}{2\pi}\int_{\partial(U\cap\mathbb{C}_j)}S_L^{-1}(s,q)ds_jf(s)\text{, for every } q \in U.$$
If $f$ is a right slice hyperholomorphic function on a set that contains $\overline{U}$, then
$$f(q)=\frac{1}{2\pi}\int_{\partial(U\cap\mathbb{C}_j)}f(s)ds_jS_R^{-1}(s,q)\text{, for every } q \in U.$$
These integrals depend neither on $U$ nor on the imaginary unit $j\in\mathbb{S}$.
\end{theorem}
\begin{definition}[{\cite[Definition 2.3.1]{31}}]
Let $(X,+)$ be an abelian group.
\begin{itemize}
\item $X$ is a right quaternionic vector space denoted by $X_R$ if it is endowed with a right quaternionic multiplication $(X,\mathbb{H})\rightarrow X$, $(u, q) \mapsto uq$ such that for all $u,v \in X$ and all $p, q \in\mathbb{H}$,
$$u(p + q) = up + uq, (u + v)q = uq + vq, (up)q = u(pq) \text{ and } u1 = u.$$
\item $X$ is a left quaternionic vector space denoted by $X_L$ if it is endowed with a left quaternionic multiplication $(\mathbb{H},X)\rightarrow X$, $(q,u) \mapsto qu$ such that for all $u,v \in X$ and all $p, q \in\mathbb{H}$,
$$(p + q)u = pu + qu, q(u + v) = qu + qv, q(pu) = (qp)u \text{ and } 1u = u.$$
\item $X$ is a two-sided quaternionic vector space if it is endowed with a left and a right quaternionic  multiplication such that $X$ is both a left and a right quaternionic vector space and such that $ru = ur$ for all $r\in\mathbb{R}$, and $(pu)q=p(uq)$ for all $p,q\in\mathbb{H}$ and all $u\in X$.
\end{itemize}
\end{definition}
\begin{definition}
Let $X_R$ be a right quaternionic vector space. A function  $\Vert\cdot\Vert : X_R \rightarrow [0;+\infty)$ is called a norm on $X_R$, if it satisfies
\begin{itemize}
\item[(i)] $\Vert u\Vert=0$ if and only if $u=0$;
\item[(ii)] $\Vert uq\Vert=\Vert u\Vert \vert q\vert$ for all $u\in X_R$ and all $q \in\mathbb{H}$;
\item[(iii)] $\Vert u+v\Vert\leq \Vert u\Vert+\Vert v\Vert$ for all $u,v\in X_R$.
\end{itemize}
If $X_R$ is complete with respect to the metric induced by $\Vert\cdot\Vert$, we call $X_R$ a right quaternionic Banach space.\\
Let $X_L$ be a left quaternionic vector space. A function $\Vert\cdot\Vert : X_L \rightarrow [0;+\infty)$ is called a norm on $X_L$, if it satisfies $(i)$, $(iii)$ and 
\begin{itemize}
\item[(ii')] $\Vert qu\Vert=\vert q\vert\Vert u\Vert$ for all $u\in X_L$ and $q \in\mathbb{H}$.
\end{itemize}
If $X_L$ is complete with respect to the metric induced by $\Vert\cdot\Vert$, we call $X_L$ a left quaternionic Banach space.\\
Finally, a two-sided quaternionic vector space $X$ is called a two-sided quaternionic Banach space if it is endowed with a norm $\Vert\cdot\Vert$ such that it is both a left and a right quaternionic Banach space.
\end{definition}
\begin{remark}
\emph{If $X$ is a two-sided quaternionic Banach space, then  $\Vert qu\Vert=\Vert uq\Vert=\vert q\vert\Vert u\Vert$ for all $u\in X$ and all $q \in\mathbb{H}$.}
\end{remark}
\begin{definition}[{\cite[Definition 2.3.9]{31}}, Slice hyperholomorphic vector-valued functions]
Let $U\subseteq\mathbb{H}$ be an axially symmetric open set and let $\mathcal{U}=\{(u,v)\in\mathbb{R}^2: u+v\mathbb{S}\subset U\}$. A function $f : U\rightarrow X_L$ with values in a left quaternionic Banach space $X_L$ is called a left slice function if it is of the form:
$$f(q) = f_0(u, v) + jf_1(u, v) \text{ for every } q = u + vj \in U$$
with two functions $f_0, f_1 : \mathcal{U}\rightarrow X_L$ that satisfy the compatibility conditions $(\ref{3eq6})$.\\
If in addition $f_0$ and $f_1$ satisfy the Cauchy-Riemann equations $(\ref{3eq7})$, then $f$ is called left slice hyperholomorphic.\\
A function $f : U\rightarrow X_R$ with values in a right quaternionic Banach space $X_R$ is called a right slice function if it is of the form:
$$f(q) = f_0(u, v) + f_1(u, v)j \text{ for every } q = u + vj \in U$$
with two functions $f_0, f_1 : \mathcal{U}\rightarrow X_R$ that satisfy the compatibility conditions $(\ref{3eq6})$.\\
If in addition $f_0$ and $f_1$ satisfy the Cauchy-Riemann equations $(\ref{3eq7})$, then $f$ is called right slice hyperholomorphic.
\end{definition}
\begin{theorem}[{\cite[Theorem 2.3.19]{31}}, Vector-valued Cauchy formulas]
Let $U \subset\mathbb{H}$ be a bounded slice Cauchy domain, let $j\in\mathbb{S}$, and set $ds_j = ds(-j)$. If $f$ is a left slice hyperholomorphic function with values in a left quaternionic Banach space $X_L$ that is defined on a set that contains $\overline{U}$, then
$$f(q)=\frac{1}{2\pi}\int_{\partial(U\cap\mathbb{C}_j)}S_L^{-1}(s,q)ds_jf(s)\text{, for every } q \in U.$$
If $f$ is a right slice hyperholomorphic function with values in a right quaternionic Banach space $X_R$ that is defined on a set that contains $\overline{U}$, then
$$f(q)=\frac{1}{2\pi}\int_{\partial(U\cap\mathbb{C}_j)}f(s)ds_jS_R^{-1}(s,q)\text{, for every } q \in U.$$
These integrals depend neither on $U$ nor on the imaginary unit $j\in\mathbb{S}$.
\end{theorem}
\begin{definition}
Let $X$ be a two-sided quaternionic Banach space. A right $($resp. left$)$ linear operator on $X$ is a map $T : X \rightarrow X$ such that:
$$T(up + v) = (Tu)p + Tv (\text{ resp. } T(pu + v) = p(Tu) + Tv)  \text{ for all } u, v \in X \text{ and all } p\in\mathbb{H}.$$
\end{definition}
A right or left linear operator $T$ on $X$ is called bounded if
$$\Vert T\Vert:=\sup\{\Vert Tu\Vert: u\in X, \Vert u\Vert=1\}<\infty.$$
The set of all right $($resp. left$)$ linear bounded operators on $X$ is denoted by $\mathcal{B}_R(X)$ $($resp. $\mathcal{B}_L(X))$. $ \mathcal{B}_R(X)$ $($resp. $\mathcal{B}_L(X))$ is viewed as a two-sided quaternionic vector space equipped with the metric $\mathcal{B}_R(X)\times\mathcal{B}_R(X)\ni(A,B)\mapsto\Vert A-B\Vert$ $($resp. $\mathcal{B}_L(X)\times\mathcal{B}_L(X)\ni(A,B)\mapsto\Vert A-B\Vert)$.\\

In a two-sided quaternionic Banach space $X$, we can define a left and a right quaternionic  multiplication on $\mathcal{B}_R(X)$ $($resp. $\mathcal{B}_L(X))$ by
$$(Tq)u = T(qu)\text{ and } (qT)(u) = q(T(u)) \text{ for all } q \in\mathbb{H}, u\in X  \text{ and all } T \in \mathcal{B}_R(X)$$
$$(\text{resp. } (Tq)u = T(u)q\text{ and } (qT)(u) = T(uq) \text{ for all } q \in\mathbb{H},u\in X  \text{ and all } T \in \mathcal{B}_L(X)).$$

The spectral theory over quaternionic Hilbert spaces has been developed in \cite{34} and \cite{37}.\\

In the remainder of this paper, $X$ will be a two-sided quaternionic Banach space. We will consider just right linear operators on $X$. The theory we develop here also applies in the case of left linear operators  with obvious modifications.\\

Let $T \in \mathcal{B}_R(X)$. We want to define the notion of spectrum for right linear operators on $X$ such that this notion generalize the known results on the spectrum in the complex case (for instance, the compactness of the spectrum, the spectrum of self-adjoint operators is real).  Note that if $q\in\mathbb{H}$, $T-Iq=T-qI$, where $I$ is the identity operator on $X$. Take $X=\mathbb{H}\oplus\mathbb{H}$ equipped with the standard scalar product:
$$\langle \begin{bmatrix}
p \\ 
q
\end{bmatrix};\begin{bmatrix}
p' \\ 
q'
\end{bmatrix}\rangle=\bar{p}p'+\bar{q}q'\text{ for all }p,q,p',q'\in\mathbb{H}.$$
Then $T:=\begin{bmatrix}
0 &\mathrm{i} \\ 
-\mathrm{i} & 0
\end{bmatrix}$ is self-adjoint. Let $u=\begin{bmatrix}
1 \\ 
-\mathrm{k}
\end{bmatrix}$, then $(T-\mathrm{j}I)u=0$, hence $\mathrm{j}$ is an eigenvalue of $T$, thus the spectrum of $T$ is not real, and so the operators $T-Iq$ and $T-qI$ should not be used to define the spectrum of $T$. F. Colombo et al. \cite{32} extended the definitions of the spectrum and resolvent in quaternionic Banach spaces as follows.
\begin{definition}
Let $T \in \mathcal{B}_R(X)$. For $q \in\mathbb{H}$, we set
$$Q_q(T) := T^2 -2\text{Re}(q)T + \vert q\vert^2I.$$
Where $I$ is the identity operator on $X$.
We define the S-resolvent set $\rho_S(T)$ of $T$ as:
$$\rho_S(T) := \{q \in\mathbb{H} : Q_q(T) \text{ is invertible in } \mathcal{B}_R(X)\},$$
and we define the S-spectrum $\sigma_S(T)$ of $T$ as:
$$\sigma_S(T):=\mathbb{H}\setminus\rho_S(T).$$
\end{definition}
\begin{proposition}[{\cite[Proposition 3.1.8]{31}}]
Let $T\in\mathcal{B}_R(X)$. The sets $\sigma_S(T)$ and $\rho_S(T)$ are axially symmetric.
\end{proposition}
\begin{theorem}[{\cite[Theorem 3.1.13]{31}}, Compactness of the S-spectrum]
Let $T\in\mathcal{B}_R(X)$. The S-spectrum $\sigma_S(T)$ of $T$ is a nonempty compact set contained in the closed ball $\{q\in\mathbb{H}:\vert q\vert\leq\Vert T\Vert\}$.
\end{theorem}
Let $T \in \mathcal{B}_R(X)$. Then the S-spectral radius of $T$ is defined to be the nonnegative real number
$$r_S(T) := \sup\{\vert q\vert : q \in \sigma_S(T)\}.$$
\begin{theorem}[{\cite[Theorem 4.2.3]{31}}]
\label{3t6}
For $T \in \mathcal{B}_R(X)$, we have
$$r_S(T) = \lim_{n\rightarrow+\infty}\Vert T^n\Vert^{\frac{1}{n}}.$$
\end{theorem}
\begin{theorem}
\label{3t3}
Let $T\in\mathcal{B}_R(X)$ and $q\in\mathbb{H}$ with $r_S(T) < \vert q\vert$. Then
$$(T^2-2\text{Re}(q)T + \vert q\vert^2I)^{-1}=\sum_{n=0}^{+\infty}T^n\sum_{k=0}^{n}\bar{q}^{-k-1}q^{-n+k-1},$$
where this series converges in the operator norm.
\end{theorem}
\begin{proof}
Let 
$$a_n:=\sum_{k=0}^{n}\bar{q}^{-k-1}q^{-n+k-1},$$
then 
$$\vert a_n\vert\leq (n+1)\vert q\vert^{-n-2}.$$
Hence 
$$\Vert T^n\sum_{k=0}^{n}\bar{q}^{-k-1}q^{-n+k-1}\Vert\leq \Vert T^n\Vert (n+1)\vert q\vert^{-n-2}.$$
We have 
$$\lim_{n\rightarrow\infty}\left( \Vert T^n\Vert (n+1)\vert q\vert^{-n-2}\right)^{\frac{1}{n}}=\frac{r_S(T)}{\vert q\vert}.$$
Thus the series $\sum_{n=0}^{+\infty}T^n\sum_{k=0}^{n}\bar{q}^{-k-1}q^{-n+k-1}$ converges in the operator norm. The rest follows
from the proof of \cite[Theorem 3.1.5]{31}.
\end{proof}
\begin{definition}[{\cite[Definition 3.2.5]{31}}, S-functional calculus]
Let $T \in\mathcal{B}_R(X)$. Let $U \supset\sigma_S(T)$ be a bounded slice Cauchy domain, let $j\in\mathbb{S}$, and set $ds_j = ds(-j)$. For every $f\in\mathcal{SH}_L(\bar{U})$, we define
$$f(T):=\frac{1}{2\pi}\int_{\partial(U\cap\mathbb{C}_j)}S_L^{-1}(s,T)ds_jf(s).$$
For every $f\in\mathcal{SH}_R(\bar{U})$, we define
$$f(T):=\frac{1}{2\pi}\int_{\partial(U\cap\mathbb{C}_j)}f(s)ds_jS_R^{-1}(s,T).$$
\end{definition}
Let $K\subseteq\mathbb{H}$. In the following, we mean by $\mathcal{SH}_R(K)$ $($resp. $\mathcal{SH}_L(K)$, $\mathcal{N}(K))$, the set of all right $($resp. left, intrinsic$)$ slice hyperholomorphic functions on an open axially symmetric set $U$ that contains $K$.
\begin{theorem}[{\cite[Theorem 4.1.3]{31}}, Product rule]
\label{3t7}
Let $T \in\mathcal{B}_R(X)$, $f\in\mathcal{N}(\sigma_S(T))$ and $g\in\mathcal{SH}_L(\sigma_S(T))$ or $g\in\mathcal{SH}_R(\sigma_S(T))$. Then
$$(fg)(T) = f(T)g(T).$$
\end{theorem}
\begin{theorem}[{\cite[Theorem 4.2.1]{31}}, The spectral mapping theorem]
\label{3t8}
Let $T \in\mathcal{B}_R(X)$ and $f\in\mathcal{N}(\sigma_S(T))$. Then
$$\sigma_S(f(T))=f(\sigma_S(T)):=\{f(q):q\in\sigma_S(T)\}.$$
\end{theorem}
\begin{theorem}[{\cite[Theorem 4.2.4]{31}}, Composition rule]
\label{3t9}
Let $T \in\mathcal{B}_R(X)$ and $f\in\mathcal{N}(\sigma_S(T))$. If $g\in\mathcal{SH}_L(\sigma_S(f(T)))$, then $g\circ f\in\mathcal{SH}_L(\sigma_S(T))$, and if $g\in\mathcal{SH}_R(f(\sigma_S(T)))$, then $g\circ f\in\mathcal{SH}_R(\sigma_S(T))$. In both cases,
$$g(f(T)) = (g \circ f)(T).$$
\end{theorem}
A bounded right projection $P\in\mathcal{B}_R(X)$ (or simply projection when no confusion can arise) is such that $P^2=P$. If $P$ is a projection, then so is $I-P$, and their null spaces and ranges are related as follows:
$$\mathcal{R}(P) = \mathcal{N}(I-P) \text{ and } \mathcal{N}(P) = \mathcal{R}(I-P).$$
The range and the kernel form a pair of algebraic complements,
$$\mathcal{R}(P) + \mathcal{N}(P) = X \text{ and } \mathcal{R}(P)\cap \mathcal{N}(P)=\{0\}.$$
\section{Generalized inverse}
In this section, we study the generalized invertibility of right linear operators on quaternionic Banach spaces.
\begin{definition}
An operator $B\in\mathcal{B}_R(X)$ is called a generalized inverse of $A\in\mathcal{B}_R(X)$ if $ABA=A$ and $BAB=B$.
\end{definition}
\begin{remark}
\label{3r1}
\emph{Let $A,B\in\mathcal{B}_R(X)$.
\begin{itemize}
\item[1)] If $B$ is a generalized inverse of $A$, then $AB$ and $BA$ are projections. Indeed 
$$(AB)^2 = (ABA)B = AB\text{, } (BA)^2=B(ABA)=BA.$$
Clearly, $\mathcal{R}(AB)\subseteq\mathcal{R}(A)$ and $\mathcal{N}(A)\subseteq\mathcal{N}(BA)$. Since $ABA=A$, we have $\mathcal{R}(A)=\mathcal{R}(ABA)\subseteq\mathcal{R}(AB)$ and $\mathcal{N}(BA)\subseteq\mathcal{N}(ABA)=\mathcal{N}(A)$. It follows that $\mathcal{R}(AB)=\mathcal{R}(A)$ and $\mathcal{N}(BA)=\mathcal{N}(A)$. Similarly, $\mathcal{R}(BA)=\mathcal{R}(B)$ and $\mathcal{N}(AB)=\mathcal{N}(B)$. Since $X=\mathcal{R}(AB)\oplus \mathcal{N}(AB)$, $X=\mathcal{R}(A)\oplus \mathcal{N}(B)$.
\item[2)] If $ABA=A$, then $T:=BAB$ is a generalized inverse of $A$.
\item[3)] If $A$ is left $($resp. right$)$ invertible, then it is generalized invertible.
\end{itemize}}
\end{remark}
\begin{theorem}
An operator $A\in\mathcal{B}_R(X)$ has a generalized inverse if and only if both the range $\mathcal{R}(A)$ and the null space $\mathcal{N}(A)$ are closed complemented subspaces of $X$.
\end{theorem}
\begin{proof}
Let $Y$ and $Z$ be two closed complemented subspaces of $X$. If $X=\mathcal{R}(A)\oplus Y=\mathcal{N}(A)\oplus Z$, then the operator defined by $B(Az+y)=z$ with $z\in Z$ is right linear on $X$. It is easy to check that $B$ is a generalized inverse of $A$.\\
Conversely, if $A$ has a generalized inverse $B\in\mathcal{B}_R(X)$, then by Remark \ref{3r1}, $\mathcal{R}(AB)=\mathcal{R}(A)$ and $\mathcal{N}(BA)=\mathcal{N}(A)$. Since $AB$ and $BA$ are bounded right projections, both the range $\mathcal{R}(AB)$ and the null space $\mathcal{N}(BA)$ are closed complemented subspaces of $X$, then so is $\mathcal{R}(A)$ and $\mathcal{N}(A)$.
\end{proof}
The generalized inverse is not unique in general, the following theorem describes all generalized inverses of $A\in\mathcal{B}_R(X)$. 
\begin{theorem}
\label{3t4}
Suppose $B\in\mathcal{B}_R(X)$ is a generalized inverse of $A\in\mathcal{B}_R(X)$. Then the set of all generalized inverses of $A$ consists of all operators of the form:
$$ T= PBQ,$$
where $Q$ is a projection onto $\mathcal{R}(A)$ and $P$ is a projection whose kernel coincides with $\mathcal{N}(A)$. 
\end{theorem}
\begin{proof}
Let $B$ and $T$ be two generalized inverses of $A$, then $ABA=ATA=A$ and $TAT=T$, hence $T(ABA)T=(TA)B(AT)=T$. By Remark \ref{3r1}, $TA$ and $AT$ are projections and such that $\mathcal{N}(TA)=\mathcal{N}(A)$ and $\mathcal{R}(A)=\mathcal{R}(AT)$.\\
Conversely, if $B$, $P$ and $Q$ are as in Theorem \ref{3t4}, then $QA=A=AP$, hence $A(PBQ)A=A$ and $(PBQ)A(PBQ)=PBQ$.
\end{proof}
\begin{lemma}
\label{3l2}
Let $A\in\mathcal{B}_R(X)$. Then $A$ has a generalized inverse $B$ such that $AB=BA$ if and only if $X$ can be written as $X = \mathcal{R}(A)\oplus\mathcal{N}(A)$. In such a case, $B$ is unique.
\end{lemma}
\begin{proof}
By Remark \ref{3r1}, $X=\mathcal{R}(A)\oplus \mathcal{N}(B)$ and since $AB=BA$, $\mathcal{N}(B)=\mathcal{N}(AB)=\mathcal{N}(BA)=\mathcal{N}(A)$, then $X=\mathcal{R}(A)\oplus \mathcal{N}(A)$.\\
Conversely, since $X = \mathcal{R}(A)\oplus\mathcal{N}(A)$, the restriction $T:=A/\mathcal{R}(A)$ is invertible, then $B=T^{-1}\oplus 0$ is a generalized inverse of $A$ which commutes with $A$. If $B$ and $C$ are two generalized inverses of $A$ that commute with $A$, then $B=AB^2=CA^2B^2=CAB=C^2A^2B=C^2A=C$.
\end{proof}
\begin{theorem}
Suppose $A\in\mathcal{B}_R(X)$ with generalized inverse $B$ such that $AB=BA$. Then 
$$\sigma_S(B)\setminus\{0\}=\{q^{-1}:q\in\sigma_S(A)\setminus\{0\}\}.$$
\end{theorem}
\begin{proof}
By Lemma \ref{3l2}, $X = \mathcal{R}(A)\oplus\mathcal{N}(A)$, then $A=T\oplus 0$ on $\mathcal{R}(A)\oplus\mathcal{N}(A)$ and $B=T^{-1}\oplus 0$. We have $Q_q(B)=Q_q(T^{-1})\oplus Q_q(0)$, for all $q\in\mathbb{H}$. Then we have $\sigma_S(B)=\sigma_S(T^{-1})\cup\sigma_S(0)$, since $Q_q(0)=\vert q\vert^2 I$ is always invertible where $I$ is the identity operator on $\mathcal{N}(A)$,
$$\sigma_S(B)\setminus\{0\}=\sigma_S(T^{-1})\setminus\{0\}.$$
The function $f:\mathbb{H}\setminus\{0\}\ni q\mapsto q^{-1}$ is intrinsic slice hyperholomorphic (because $q^{-1}=\frac{\bar{q}}{\vert q\vert^2}$), then by Theorem \ref{3t8}, $\sigma_S(T^{-1})=\sigma_S(f(T))=\{q^{-1}:q\in\sigma_S(T)\}$. Thus
$$\sigma_S(B)\setminus\{0\}=\{q^{-1}:q\in\sigma_S(A)\setminus\{0\}\}.$$
\end{proof}
\section{Drazin inverse}
In this section, we study the Drazin invertibility of right linear operators acting on a quaternionic Banach space.
\begin{definition}
\label{3d1}
Let $A\in\mathcal{B}_R(X)$. An element $B\in\mathcal{B}_R(X)$ is a Drazin inverse of $A$, written $B = A^d$, if
\begin{equation}
 AB=BA\text{, } AB^2=B\text{, } A^{k+1}B=A^k,\label{3eq0}
\end{equation}
for some nonnegative integer $k$. The least nonnegative integer $k$ for which these equations hold is the Drazin index $i(A)$ of $A$.
\end{definition}
\begin{definition}
An element $A$ of $\mathcal{B}_R(X)$ is called quasinilpotent if $\sigma_S(A)=\{0\}$. The set of all quasinilpotent elements in $\mathcal{B}_R(X)$ will be denoted by $QN(\mathcal{B}_R(X))$.
\end{definition}
\begin{proposition}
An element $A$ of $\mathcal{B}_R(X)$ is quasinilpotent if and only if, for every $T$ commuting with $A$, we have $I-TA$ is invertible.
\end{proposition}
\begin{proof}
Let $A\in\mathcal{B}_R(X)$, assume that for every $T\in\mathcal{B}_R(X)$ commuting with $A$, we have $I-TA$ is invertible. Let $T=\frac{-1}{\vert q\vert^2}A+\frac{2\text{Re}(q)}{\vert q\vert^2}I$ with $q\in\mathbb{H}\setminus\{0\}$, clearly $T$ commutes with $A$ and $I-TA=\frac{1}{\vert q\vert^2}[A^2-2\text{Re}(q)A+\vert q\vert^2I]$ is invertible, hence $\sigma_S(A)=\{0\}$.\\
Conversely, if $\sigma_S(A)=\{0\}$. Let $T\in\mathcal{B}_R(X)$ commutes with $A$, then by Theorem \ref{3t6}, $r_S(TA)\leq r_S(T)r_S(A)=0$ and hence $\sigma_S(TA)=\{0\}$. Then by Theorem \ref{3t8}, $\sigma_S(I-TA)=\{1\}$ and hence $I-TA$ is invertible.
\end{proof}
An operator $A\in\mathcal{B}_R(X)$ is said to be nilpotent if there exists $k\in\mathbb{N}$ such that $A^k=0$. The least nonnegative integer $k$ for which $A^k=0$ is called the nilpotency index of $A$ and the set of all nilpotent elements in $\mathcal{B}_R(X)$ is denoted by $N(\mathcal{B}_R(X))$.
\begin{lemma}
In $\mathcal{B}_R(X)$, $(\ref{3eq0})$ is equivalent to
\begin{equation}
AB=BA, AB^2=B, A-A^2B\in N(\mathcal{B}_R(X)).\label{3eq1}
\end{equation}
The Drazin index $i(A)$ is equal to the nilpotency index of $A-A^2B$.
\end{lemma}
\begin{proof}
If $AB=BA$ and $AB^2=B$, then $I-AB$ is a projection. Hence the equivalence and the last statement are given by this equalities $(A-A^2B)^k=A^k(I-AB)^k=A^k(I-AB)=A^k-A^{k+1}B$.
\end{proof}
Koliha \cite[Definition 2.3]{36} generalized the notion of Drazin invertibility in a complex Banach algebra. According to this definition one can generalize the notion of Drazin invertibility in $\mathcal{B}_R(X)$.
\begin{definition}
Let $A\in\mathcal{B}_R(X)$. An element $B\in\mathcal{B}_R(X)$ is a generalized Drazin inverse of $A$, written $B = A^D$, if
\begin{equation}
AB=BA, AB^2=B, A-A^2B\in QN(\mathcal{B}_R(X)). \label{3eq2}
\end{equation}
\end{definition}
\begin{lemma}
\label{3l1}
In $\mathcal{B}_R(X)$, an element $A$ has a Drazin $($resp. generalized Drazin$)$ inverse if and only if there is a projection $P$ commuting with $A$ such that
\begin{equation}
AP\in N(\mathcal{B}_R(X))(\text{resp. }AP\in QN(\mathcal{B}_R(X)))\text{ and } A +P\text{ is invertible}.\label{3eq3}
\end{equation}
A Drazin $($resp. generalized Drazin$)$ inverse of $A$ is given by
\begin{equation}
A^d = (A+P)^{-1}(I-P)(\text{resp. } A^D = (A+P)^{-1}(I-P)) \label{3eq4}.
\end{equation}
\end{lemma}
\begin{proof}
Suppose that there is a projection $P$ commuting with $A$ and satisfying (\ref{3eq3}). Set $B=(A+P)^{-1}(I-P)$, then $AB=BA$, $AB^2=B$ and $A-A^2B=AP\in N(\mathcal{B}_R(X))\;(\text{resp. }A-A^2B=AP\in QN(\mathcal{B}_R(X)))$.\\
Conversely, suppose that $B$ satisfies (\ref{3eq1}) (resp. (\ref{3eq2})) and set $P=I-AB$. Since $AB^2=B$, $P$ is a projection commuting with $A$ and $AP=A-A^2B$, then $AP\in N(\mathcal{B}_R(X))\text{ (resp. }AP\in QN(\mathcal{B}_R(X))\text{)}$. Furthermore $(A+P)(B+P)=(B+P)(A+P)=I+AP$ and $I+AP$ is invertible because $\sigma_S(I+AP)=\{1\}$, then $A+P$ is invertible in $\mathcal{B}_R(X)$.
\end{proof}
\begin{definition}
Let $A\in \mathcal{B}_R(X)$. For $s\in\rho_S(T)$, we define the left S-resolvent
operator as
$$S^{-1}_L(s,A)=-Q_s(T)^{-1}(T-\bar{s} I).$$
\end{definition}
\begin{theorem}[{\cite[Theorem 4.1.5]{31}}]
\label{3t10}
Let $A\in \mathcal{B}_R(X)$ and assume that $\sigma_S(A)=\sigma_1\cup\sigma_2$ with
$$dist(\sigma_1, \sigma_2) > 0.$$
We choose an open axially symmetric set $O$ with $\sigma_1 \subset O$ and $\overline{O}\cap\sigma_2 =\emptyset$, and
define a function $\chi_{\sigma_1}$ on $\mathbb{H}$ by $\chi_{\sigma_1}(s) = 1$ for $s \in O$ and $\chi_{\sigma_1}(s) = 0$ for $s \notin O$. Then $\chi_{\sigma_1}\in \mathcal{N}(\sigma_S(A))$, and for an arbitrary imaginary unit $j$ in $\mathbb{S}$ and an arbitrary bounded slice Cauchy domain $U\subset\mathbb{H}$ such that $\bar{U} \subset O$, we have
$$P_{\sigma_1}:=\chi_{\sigma_1}(A)=\frac{1}{2\pi}\int_{\partial(U\cap\mathbb{C}_j)}S^{-1}_L(s,A)ds_j$$
is a continuous projection that commutes with $A$. Hence $P_{\sigma_1}(X)$ is a right linear
subspace of $X$ that is invariant under $A$.
\end{theorem}
\begin{remark}
\emph{Let $q\in\mathbb{H}$. If $\sigma_1=\{q\}$, we say that the projection $P_{\sigma_1}$ is the Riesz's projection of $A$ corresponding to $q$.}
\end{remark}
Denote by $acc\;U$ (resp. $iso\;U$) the set of all accumulation (resp. isolated) points of a set $U\subseteq\mathbb{H}$.
\begin{theorem}
\label{3t1}
Let $A\in\mathcal{B}_R(X)$. Then $0\notin acc\;\sigma_S(A)$ if and only if there is a projection $P\in\mathcal{B}_R(X)$
commuting with $A$ such that
\begin{equation}
AP\in QN(\mathcal{B}_R(X))\text{ and } A+P \text{ is invertible in } \mathcal{B}_R(X). \label{3eq5}
\end{equation}
Moreover, $0 \in iso\;\sigma_S(A)$ if and only if $P\neq 0$, in which a case $P$ is the Riesz's projection of $A$
corresponding to $q = 0$.
\end{theorem}
\begin{proof}
Clearly, $0\notin\sigma_S(A)$ if and only if (\ref{3eq5}) holds with $P=0$.\\
Assume that $0 \in iso\;\sigma_S(A)$. Let $P$ be the spectral projection of $A$ corresponding to $q = 0$, then $P\neq 0$, commutes with $A$ and $AP=id(A)\chi_{\{0\}}(A)=(id\chi_{\{0\}})(A)$ where $id : \mathbb{H}\rightarrow\mathbb{H},q\mapsto q$. Hence $\sigma_S(AP)=id\chi_{\{0\}}(\sigma_S(A))=\{0\}$, thus $AP\in QN(\mathcal{B}_R(X))$. Similarly $A+P=id(A)+\chi_{\{0\}}(A)=(id+\chi_{\{0\}})(A)$, then $0\notin \sigma_S(A+P)=(id+\chi_{\{0\}})\sigma_S(A)$, hence $A+P$ is invertible.\\
Conversely, assume that there is a nonzero projection $P$ commuting with $A$ such that (\ref{3eq5}) holds. For any $q\in\mathbb{H}$, we have 
\begin{align*}
&A^2-2\text{Re}(q)A+\vert q\vert^2 I\\
&=P((AP)^2-2\text{Re}(q)AP+\vert q\vert^2 I)+(I-P)((A+P)^2-2\text{Re}(q)(A+P)+\vert q\vert^2 I).
\end{align*}
There is $r>0$ such that if $\vert q\vert<r$ then $(A+P)^2-2\text{Re}(q)(A+P)+\vert q\vert^2 I$ is invertible. Since $AP\in QN(\mathcal{B}_R(X))$, $(AP)^2-2\text{Re}(q)AP+\vert q\vert^2 I$ is invertible of all $q\neq0$. Hence for all $0<\vert q\vert<r$, it is easy to check that 
\begin{align*}
(A^2-2\text{Re}(q)A+\vert q\vert^2 I)^{-1}=&P((AP)^2-2\text{Re}(q)AP+\vert q\vert^2 I)^{-1}\\
&+(I-P)((A+P)^2-2\text{Re}(q)(A+P)+\vert q\vert^2 I)^{-1}.
\end{align*}
That is,
$$Q_q(A)^{-1}=PQ_q(AP)^{-1}+(I-P)Q_q(A+P)^{-1}.$$
Since $S^{-1}_L(q,A)=-Q_q(A)^{-1}(A-\bar{q}I)$, it is easy to see that
\begin{equation}
S^{-1}_L(q,A)=PS^{-1}_L(q,AP)+(I-P)S^{-1}_L(q,A+P).\label{3eq8}
\end{equation}
Since $P\neq 0$, $0 \in iso\;\sigma_S(A)$. Indeed, if $A$ is invertible, then $A^{-1}AP=P$, so that $r_S(A^{-1}AP)\leq r_S(A^{-1})r_S(AP)=0$ and $r_S(P)=1$. To show that $P$ is the Riesz's projection of $A$ corresponding to $q=0$. Let $j$ and $U$ as in Theorem \ref{3t10}, then
$$\chi_{\{0\}}(A)=\frac{1}{2\pi}\int_{\partial(U\cap\mathbb{C}_j)}S^{-1}_L(s,A)ds_j.$$
If we take $U$ such that $U\subset\{q\in\mathbb{H}:\vert q\vert<\frac{r}{2}\}$, then by (\ref{3eq8})
\begin{align*}
\chi_{\{0\}}(A)&=\frac{1}{2\pi}\int_{\partial(U\cap\mathbb{C}_j)}S^{-1}_L(s,A)ds_j\\
&=\frac{1}{2\pi}\int_{\partial(U\cap\mathbb{C}_j)}PS^{-1}_L(s,AP)+(I-P)S^{-1}_L(s,A+P)ds_j\\
&=\frac{1}{2\pi}\int_{\partial(U\cap\mathbb{C}_j)}PS^{-1}_L(s,AP)ds_j+\frac{1}{2\pi}\int_{\partial(U\cap\mathbb{C}_j)}(I-P)S^{-1}_L(s,A+P)ds_j\\
&=P\frac{1}{2\pi}\int_{\partial(U\cap\mathbb{C}_j)}S^{-1}_L(s,AP)ds_j+(I-P)\frac{1}{2\pi}\int_{\partial(U\cap\mathbb{C}_j)}S^{-1}_L(s,A+P)ds_j.
\end{align*}
Since $S^{-1}_L(\cdot,A+P)$ is right slice hyperholomorphic function on $U$ (see, \cite[Lemma 3.1.11]{31}), 
$$\int_{\partial(U\cap\mathbb{C}_j)}S^{-1}_L(s,A+P)ds_j=0.$$
On the other hand, 
$$\frac{1}{2\pi}\int_{\partial(U\cap\mathbb{C}_j)}S^{-1}_L(s,AP)ds_j=I,$$
because $\sigma_S(AP)=\{0\}\subset U$. Hence $\chi_{\{0\}}(A)=P$. This completes the proof.
\end{proof}
\begin{corollary}
The Drazin $($resp. generalized Drazin$)$ inverse of an operator $A\in\mathcal{B}_R(X)$ is uniquely determined.
\end{corollary}
\begin{proof}
If $A$ is invertible, then it has a unique Drazin inverse which coincides with its inverse $A^{-1}$. Assume that $A$ is not invertible and let $B$ and $C$ be two Drazin (resp. generalized Drazin) inverses of $A$, then by the proof of Lemma \ref{3l1}, $I-BA$ and $I-AC$ are two projections commuting with $A$ and satisfying (\ref{3eq5}). Then by Theorem \ref{3t1}, $I-BA=I-AC$, and so $BA=AC$, thus $B=B^2A=BAC=AC^2=C$. Hence the Drazin (resp. generalized Drazin) inverse of $A$ is unique.
\end{proof}
\begin{theorem}
\label{3t2}
Let $A\in\mathcal{B}_R(X)$. If $0\in iso\;\sigma_S(A)$, then 
$$A^D=f(A),$$
where $f \in\mathcal{N}(\sigma_S(A))$ is such that $f$ is $0$ in an axially symmetric neighborhood of $0$ and $f(q) = q^{-1}$ in an axially symmetric neighborhood of $\sigma_S(A)\setminus\{0\}$, and
$$\sigma_S(A^D)\setminus\{0\}=\{q^{-1}:q\in\sigma_S(A)\setminus\{0\}\}.$$
\end{theorem}
\begin{proof}
Let $O_1$ be an axially symmetric open neighborhood of $0$ and $O_2$ be an axially symmetric open neighborhood of $\sigma_S(A)\setminus\{0\}$ with $\overline{O_1}\cap\overline{O_2}=\emptyset$. Define $f$ by $f(q)=0$ if $q\in O_1$ and $f(q)=q^{-1}$ if $q\in O_2$, clearly $f\in \mathcal{N}(\sigma_S(A))$. By Theorem \ref{3t7} and Theorem \ref{3t8}, it is easy to see that (\ref{3eq2}) holds for $A$ and $f(A)$.\\
By Theorem \ref{3t8}, it follows that $\sigma_S(A^D)\setminus\{0\}=\sigma_S(f(A))\setminus\{0\}=\{f(q):q\in\sigma_S(A)\setminus\{0\}\}=\{q^{-1}:q\in\sigma_S(A)\setminus\{0\}\}.$
\end{proof}
\begin{theorem}
\label{3t5}
Let $A\in\mathcal{B}_R(X)$. The following conditions are equivalent:
\begin{itemize}
\item[(i)] $A$ is generalized Drazin invertible;
\item[(ii)] $0\notin acc\;\sigma_S(A)$;
\item[(iii)] $A = A_1 \oplus A_2$, where $A_1$ is invertible on some closed subspace $X_1$ of $X$ and $A_2$ is quasinilpotent on some complemented subspace $X_1$ of $X$.
\end{itemize} 
\end{theorem}
\begin{proof}
$(i)\Leftrightarrow(ii)$ Already proved in Lemma \ref{3l1} and Theorem \ref{3t1}.\\
$(i)\Rightarrow(iii)$ By Lemma \ref{3l1} there exists a projection $P=I-AA^D$ such that $AP$ is quasinilpotent and $AP=PA$, then $\mathcal{R}(P)$ and $\mathcal{N}(P)$ are invariant under $A$, that is $A\mathcal{R}(P)\subset\mathcal{R}(P)$ and $A\mathcal{N}(P)\subset\mathcal{N}(P)$. Let $u\in\mathcal{N}(P)$, then $u=AA^Du$, thus the restriction of $A$ to the kernel of $P$ is injective and surjective, and so invertible. If we write $A = A_1 \oplus A_2$ on $X = \mathcal{N}(P) \oplus \mathcal{R}(P)$, then $A_2\in\mathcal{B}_R(X_1)$ is quasinilpotent and $A_1\in\mathcal{B}_R(X_2)$ is invertible.\\
$(iii)\Rightarrow(i)$ It is easy to check that $A^D=A_1^{-1}\oplus 0$.
\end{proof}
\begin{corollary}
\label{3c1}
Let $A\in\mathcal{B}_R(X)$. The following conditions are equivalent:
\begin{itemize}
\item[(i)] $A$ is Drazin invertible;
\item[(iii)] $A = A_1 \oplus A_2$, where $A_1$ is invertible on some closed subspace $X_1$ of $X$, $A_2$ is nilpotent on some complemented subspace $X_1$ of $X$ and the nilpotency index of $A_2$ is the Drazin index of $A$.
\end{itemize}
\end{corollary}
\begin{proof}
Assume that $A$ is Drazin invertible, then by Theorem \ref{3t5} (iii), $A = A_1 \oplus A_2$ and $A^d=A_1^{-1}\oplus 0$. Hence, by (\ref{3eq0}), $A^{k+1}A^d=A^{k}$, then $A_1^k \oplus 0=A_1^k \oplus A_2^k$, thus $A_2^k=0$, so that the nilpotency index of $A_2$ is less than the Drazin index of $A$.\\
Conversely, let  $B = A_1^{-1} \oplus 0$, where $A_1$ is invertible and $A_2$ is nilpotent, then (\ref{3eq0}) holds for $A$, $B$ and the nilpotency index of $A_2$. Hence $A$ is Drazin invertible and the Drazin index of $A$ is less than the nilpotency index of $A_2$.
\end{proof}
\begin{theorem}
Suppose that $A\in\mathcal{B}_R(X)$ has the generalized Drazin inverse $A^D$. Then
\begin{itemize}
\item[(i)] $(A^k)^D = (A^D)^k$ for all $k\in\mathbb{N}$;
\item[(ii)] $(A^D)^D = A^2A^D$;
\item[(iii)] $((A^D)^D)^D = A^D$;
\item[(iv)]  $A^D(A^D)^D=AA^D$.
\end{itemize}
\end{theorem}
\begin{proof}
Let $f \in\mathcal{N}(\sigma_S(A))$ such that $f$ is $0$ in an axially symmetric neighborhood of $0$ and $f(q) = q^{-1}$ in an axially symmetric neighborhood of $\sigma_S(A)\setminus\{0\}$. By Theorem \ref{3t2}, $A^D=f(A)$. Let $k\in\mathbb{N}$ and $g_k\in \mathcal{N}(\mathbb{H})$ such that $g_k(q)=q^k$ if $q\in \mathbb{H}$. Clearly $f\circ g_k=g_k\circ f$ for all $k\in\mathbb{N}$, $f\circ f=g_2\circ f$, $f\circ f\circ f=f$ and $f(f\circ f)=g_1f$. The above assertions are easily verified by using the previous equalities, Theorem \ref{3t7} and Theorem \ref{3t9}.
\end{proof}
\begin{theorem}
Let $A, B\in\mathcal{B}_R(X)$ be commuting elements such that $A^d$ and $B^d$ exist. Then
$(AB)^d$ exists and
$$(AB)^d=A^dB^d.$$
\end{theorem}
\begin{proof}
By \cite[Theorem 1]{33}, $A,B,A^d$ and $B^d$ commute mutually, then the result follows by Definition \ref{3d1}.
\end{proof}
Given a right linear operator $T$ on $X$. $T$ is said to have finite ascent if there is an integer $k$ such that $\mathcal{N}(T^k) =\mathcal{N}(T^{k+1})$, the smallest such positive integer $k$ is called the \textit{ascent} of $T$ and denoted by $a(T)$. If there is no such integer we set $a(T):= \infty$. Analogously, $T$ is said to have finite descent if there is an integer $k$ such that $T^{k+1}(X) = T^k(X)$, the smallest such positive integer $k$ is called the \textit{descent} of $T$ and denoted by $d(T)$. If there is no such integer we set $d(T) :=  \infty$.\\

As in the complex case, we have the following result.
\begin{theorem}
Suppose that $T$ is a right linear operator on $X$ and let $k\in\mathbb{N}$. Then $ a(T) = d(T) \leq k$ if and only if, we have the decomposition
$$X =\mathcal{R}(T^k)\oplus \mathcal{N}(T^k).$$
\end{theorem}
\begin{proof}
Let $ k=a(T) = d(T)$, then $\mathcal{R}(T^k)\cap \mathcal{N}(T^k)=\{0\}$. Indeed, let $u\in\mathcal{R}(T^k)\cap \mathcal{N}(T^k)$, then there is a vector $v\in X$ such that $u=T^kv$, hence $T^ku=T^{2k}v=0$. Since $\mathcal{N}(T^{2k})=\mathcal{N}(T^k)$, $u=T^kv=0$. On the other hand, we have $\mathcal{R}(T^k)+ \mathcal{N}(T^k)=X$. Indeed, let $u\in X$, since $\mathcal{R}(T^k)=\mathcal{R}(T^{2k})$, there is a vector $v\in X$ such that $T^ku=T^{2k}v$, hence $u=T^{k}v+u-T^kv$. Thus $X =\mathcal{R}(T^k)\oplus \mathcal{N}(T^k).$\\
Conversely, let $u\in\mathcal{N}(T^{k+1})$. Since $\mathcal{R}(T^k)\cap \mathcal{N}(T^k)=\{0\}$, $T^ku=0$. Hence $\mathcal{N}(T^{k+1})=\mathcal{N}(T^{k})$. On the other hand, let $u\in\mathcal{R}(T^{k})$, then  there is a vector $v\in X$ such that $u=T^kv$. Since $v\in\mathcal{R}(T^k)+ \mathcal{N}(T^k)$, $u=T^kv\in\mathcal{R}(T^{k+1})$. Hence $\mathcal{R}(T^{k})=\mathcal{R}(T^{k+1})$. Let now $p=a(T)$ and $q=d(T)$, we can suppose that $p>0$ and $q>0$, assume that $p\leq q$, let $u\in\mathcal{R}(T^{p})$, then  there is a vector $v\in X$ such that $u=T^pv$. Since $v\in\mathcal{R}(T^q)\oplus \mathcal{N}(T^p)$, $u=T^pv\in\mathcal{R}(T^{p+q})\subseteq\mathcal{R}(T^{p+1})$.
Thus $ p=q\leq k$. Assume that $q\leq p$, let $u\in\mathcal{N}(T^{q+1})$, then $T(T^{q}u)=0$, that is $T^{q}u\in\mathcal{R}(T^{q})\cap \mathcal{N}(T)$. Since $X=\mathcal{R}(T^q)\oplus \mathcal{N}(T^p)$, $T^{q}u=0$, and so $u\in\mathcal{N}(T^{q})$. Thus $ p=q\leq k$. Hence $a(T)=d(T)\leq k$.
\end{proof}
\begin{theorem}
An operator $A$ in $\mathcal{B}_R(X)$ has a Drazin inverse if and only if it has finite ascent and descent. In such a case, the Drazin index of $A$ is equal to the common value of $a(A)$ and $d(A)$.
\end{theorem}
\begin{proof}
By Corollary \ref{3c1}, $A$ is Drazin invertible if and only if $A=A_1\oplus A_2$ with $A_1$ is invertible and $A_2$ is nilpotent. Let $k$ be the nilpotency index of $A_2$, then $k$ is the least integer such that $X=\mathcal{R}(T^k)\oplus \mathcal{N}(T^k)$, hence $a(A)=d(A)=k$. By Corollary \ref{3c1}, again, $k=i(A)$, thus $a(A)=d(A)=i(A)$.
\end{proof}
\begin{definition}
A two-sided quaternionic Banach algebra is a two-sided quaternionic Banach space $\mathcal{A}$ that is endowed with a product $\mathcal{A}\times\mathcal{A}\rightarrow\mathcal{A}$ such that:
\begin{itemize}
\item[(i)] The product is associative and distributive over the sum in $\mathcal{A}$;
\item[(ii)] one has $(qx)y = q(xy)$ and $x(yq) = (xy)q$ for all $x, y \in\mathcal{A}$ and all $q\in\mathbb{H}$;
\item[(iii)] one has $\Vert xy\Vert\leq \Vert x\Vert\Vert y\Vert$ for all $x, y \in\mathcal{A}$.
\end{itemize} 
If in addition there exists $e\in\mathcal{A}$ such that $ex=xe=x$ for all $x\in\mathcal{A}$, then $\mathcal{A}$ is called a two-sided quaternionic Banach algebra with unit.
\end{definition}
One can prove that $\mathcal{B}_R(X)$ and $\mathcal{B}_L(X)$ are two-sided quaternionic Banach algebras with unit.
\begin{definition}
Let $\mathcal{A}$ be a two-sided quaternionic Banach algebra and $a\in\mathcal{A}$. An element $b\in\mathcal{A}$ is a Drazin inverse of $a$, written $b = a^d$, if
\begin{equation}
 ab=ba\text{, } ab^2=b\text{, } a^{k+1}b=a^k,
\end{equation}
for some nonnegative integer $k$. The least nonnegative integer $k$ for which these equations hold is the Drazin index $i(a)$ of $a$.
\end{definition}
Let $\mathcal{A}$ be a two-sided quaternionic Banach algebra and $a\in\mathcal{A}$. For any $a\in\mathcal{A}$ we define the left multiplication of $a$ by $L_a(b) = ab$, for all $b\in\mathcal{A}$. Then $L_a\in\mathcal{B}_R(\mathcal{A})$, we have $\Vert L_a\Vert=\Vert a\Vert$.
\begin{theorem}
\label{3t10}
Let $\mathcal{A}$ be a two-sided quaternionic Banach algebra and $a\in\mathcal{A}$ with unit. Then $a$ is Drazin invertible if and only if $L_a$ is Drazin invertible. In such a case, $L_a^d=L_{a^d}$ and $i(L_a)=i(a)$.
\end{theorem}
\begin{proof}
Let $a\in\mathcal{A}$ such that $a$ is Drazin invertible. For every $b\in\mathcal{A}$, we have $L_aL_b=L_{ab}$, hence it is easy to check that $L_{a^d}=L_a^d$ and then $i(L_a)\leq i(a)$.\\
Conversely, assume that $L_a$ is Drazin invertible and let $b=L_a^d(e)$. Since $L_a^{k+1}L_a^d=L_a^k$, $a^{k+1}b=a^k$. Hence $L_a^{k+1}L_b=L_a^k=L_a^dL_a^{k+1}$, then by \cite[Theorem 4]{33} and its proof, $L_a^d=L_a^{k}L_b^{k+1}=L_{a^{k}b^{k+1}}$. Let $c=a^{k}b^{k+1}$, then $L_aL_c=L_cL_a\text{, } L_aL_c^2=L_c\text{, } L_a^{k+1}L_c=L_a^k$, hence $ac=ca\text{, } ac^2=c\text{, } a^{k+1}c=a^k$. Thus $a$ is Drazin invertible and then $i(a)\leq i(L_a)$.
\end{proof}

\end{document}